\documentclass[12pt, twoside]{article}
\usepackage{amsmath,amsthm,amssymb}
\usepackage{times}
\usepackage{enumerate}
\usepackage{bm}
\usepackage[all]{xy}
\usepackage{mathrsfs}
\usepackage{amscd}
\usepackage{color}
\def\titlerunning#1{\gdef\titrun{#1}}
\makeatletter
\def\author#1{\gdef\autrun{\def\and{\unskip, }#1}\gdef\@author{#1}}
\def\address#1{{\def\and{\\\hspace*{18pt}}\renewcommand{\thefootnote}{}%
		\footnote {#1}}%
	\markboth{\autrun}{\titrun}}
\makeatother
\def\email#1{e-mail: #1}
\def\subjclass#1{{\renewcommand{\thefootnote}{}%
		\footnote{\emph{Mathematics Subject Classification (2010):} #1}}}
\def\keywords#1{\par\medskip
	\noindent\textbf{Keywords.} #1}

\newtheorem{theorem}{Theorem}[section]
\newtheorem{corollary}[theorem]{Corollary}
\newtheorem{lemma}[theorem]{Lemma}
\newtheorem{proposition}[theorem]{Proposition}
\theoremstyle{definition}
\newtheorem{definition}[theorem]{Definition}
\newtheorem{remark}[theorem]{Remark}

\theoremstyle{example}
\newtheorem{example}[theorem]{Example}
\numberwithin{equation}{section}

\xyoption{all}

\def \M {\mathcal{M}}

\def \a {\alpha }
\def \b {\beta}

\def \de {\delta}
\def \De {\Delta}
\def \la {\lambda}

\def\w {\omega}
\def\Om{\Omega}

\def\pa{\partial}
\def\na {\nabla}
\def\Ga{\Gamma}

\begin{document}
\baselineskip=17pt

\titlerunning{Asymptotic behaviour of instantons on Cylinder Manifolds}
\title{Asymptotic behaviour of instantons on Cylinder Manifolds}

\author{Teng Huang}

\date{}

\maketitle

\address{T. Huang: School of Science, Nantong University, Nantong, Jiangsu, 226007, P. R. China and School of Mathematical Sciences, University of Science and Technology of China, Hefei, 230026, P.R. China; \email{htmath@ustc.edu.cn; htustc@gmail.com}}


\begin{abstract}
In this article, we study the instanton equation on the cylinder over a closed manifold $X$ which admits non-zero smooth $3$-form $P$ and $4$-form $Q$. Our results are (1) if $X$ is a \textbf{good} manifold, i.e., $P,Q$ satisfying $d\ast_{X}P=d\ast_{X}Q=0$, then the instanton with integrable curvature decays exponentially at the ends, and, (2) if $X$ is a real Killing spinor manifold, i.e., $P,Q$ satisfying $dP=4Q$ and $d\ast_{X}Q=(n-3)\ast_{X}P$, we prove that the solution of instanton equation is trivial under some mild conditions.
\end{abstract} 
\section{Introduction}
Let $M$ be an oriented smooth $n$-dimensional Riemannian manifold, $G$ be a compact Lie group and $P$ be a principal $G$-bundle on $M$, $\mathfrak{g}_{P}$ be the adjoint bundle of $P$. Let $A$ denote a connection on $P$ with the curvature $F_{A}$. The instanton equation on $M$ can be introduced as follows. Assume there is a $4$-form $\Om$ on $M$, then an $(n-4)$-form $\ast\Om$ exists, where $\ast$ is the Hodge star operator on $M$. A connection $A$ is called an anti-self-dual instanton, when it satisfies the instanton equation
\begin{equation}\label{1.1}
\ast F_{A}+\ast\Om\wedge F_{A}=0
\end{equation}
When $n>4$, these equations can be defined on the manifold $M$ with a special holonomy group, i.e., the holonomy group $G$ of the Levi-Civita connection on the tangent bundle $TM$ is a subgroup of the group $SO(n)$. Each solution of equation(\ref{1.1}) satisfies the Yang-Mills equation. The instanton equation (\ref{1.1}) is also well-defined on a manifold $M$ with non-integrable $G$-structures, but equation (\ref{1.1}) implies that the Yang-Mills equation will have torsion.

Instantons on the higher dimension, proposed in \cite{CDFJ} and studied in \cite{RRC,DT,DS,HN,RSW}, are important both in mathematics \cite{DS,DT} and string theory \cite{Gr,GSW}. In mathematics, the articles of Donaldson-Thomas \cite{DT} and Donaldson-Segal \cite{DS} have inspired a considerable amount of work related to gauge theory in higher dimensional. S\'{a} Earp and Walpuski focus on the study of gauge theory on $G_{2}$-manifolds, they construct $G_{2}$-instanton over some $G_{2}$-manifolds \cite{Earp3,EW,Walpuski}. In string theory, the solutions of the instanton equations on cylinders over nearly K\"{a}hler $6$-manifolds and nearly parallel $G_{2}$-manifolds have been constructed \cite{BILL,HILP,ID,ILPR}. In \cite{ILPR} Section $4$, they confirm that the standard Yang-Mills functional is infinite in their solutions.

In this article, we consider the instanton $\textbf{A}$ on the cylinder manifold over a closed manifold $X$ which admits non-zero smooth $3$-form $P$ and $4$-form $Q$. The $4$-form $\Om$ on $Cyl(X):=(\mathbb{R}\times X, dt^{2}+\rm{g}_{X})$ can be defined as 
$$\Om=dt\wedge P+Q.$$ 
Therefore the instanton equation on $Cyl(X)$ can be defined as \cite{BILL,HN},
\begin{equation}\label{E0}
\ast{F}_{\textbf{A}}+(\ast_{X}P+dt\wedge\ast_{X}Q)\wedge{F}_{\textbf{A}}=0,
\end{equation}
where $\ast_{X}$ is the Hodge star operator of $X$.

Chapter $4$ of \cite{Donaldson} introduces some analytic results about the asympotic behaviour of ASD connection on $4$-manifolds with tubular ends, and aims to give a complete definition of the Floer groups of a homology $3$-sphere. We know from analogous Floer-type theories those are the essential property needed to control solutions over infinite tubes.  All of the known construction methods of higher dimensional instantons automatically yield exponential decay. One can see the curvature will satisfy the $L^{2}$-integrability condition. There is a natural question, whether all $L^{2}$-integrable instantons have exponential decay. We prove that the instanton equation (\ref{1.1}) on the cylinder over a closed \textbf{good} manifold decays exponentially at the ends, See Theorem \ref{T1}.  We also consider the $\Om$-instantons on the cylinder over a closed manifold which admits real Killing spinors, i.e., the forms $P,Q$ satisfying  $dP=4Q$ and $d\ast_{X}Q=(n-3)\ast_{X}P$. In \cite{HT}, the author prove that the non-trivial solutions of instanton equations over the cylinder of the Riemannian manifolds with real Killing Spinors have infinite Yang-Mills energy. In this article, we will prove that if the energy density $\rho(\textbf{A})=0$, See Equation (\ref{E3}), then the instanton is a flat connection, See Theorem \ref{T3}.
\begin{remark}
In dimension 4, the instanton equation on a cylinder is the gradient flow for the Chern-Simons functional on the 3-dimensional cross-section. Critical points of this 3-dimensional functional are precisely flat connections. Similarly, the higher dimensional instanton equation on a cylinder over a  closed $G_{2}$-manifold or Calabi-Yau 3-fold $X$ can be regarded as the gradient flow of a Chern-Simons functional on $X$. However critical points of the higher dimensional Chern-Simons functional are instantons, not necessarily for flat connections. In fact, most examples of  instantons on asymptotically cylindrical manifolds do not have $L^{2}$ curvature because the limit connection is not flat.
\end{remark}
\section{Fundamental preliminaries}
We shall generally adhere to the now standard gauge-theory conventions and notation of Donaldson and Kronheimer \cite{Donaldson/Kronheimer}. Let $M$ be a closed, smooth, Riemannian manifold.  Throughout our article, $\Om^{p}(M,\mathfrak{g}_{P})$ denote the smooth $p$-forms with values in $\mathfrak{g}_{P}$. Given a connection on $P$, we denote by $\na_{A}$ the corresponding covariant derivative on $\Om^{\ast}(M,\mathfrak{g}_{P})$ induced by $A$ and the Levi-Civita connection of $M$. Let $d_{A}$ denote the exterior derivative associated to $\na_{A}$. For $u\in L^{p}_{k,A}(M,\mathfrak{g}_{P})$, where $1\leq p<\infty$ and $k$ is an integer, we denote
\begin{equation}\nonumber
\|u\|_{L^{p}_{k,A}(M)}:=\big{(}\sum_{j=0}^{k}\int_{M}|\na^{j}_{A}u|^{p}dvol_{\rm{g}}\big{)}^{1/p},
\end{equation}
where $\na^{j}_{A}:=\na_{A}\circ\ldots\circ\na_{A}$ (repeated $j$ times for $j\geq0$). For $p=\infty$, we denote
\begin{equation}\nonumber
\|u\|_{L^{\infty}_{k,A}(M)}:=\sum_{j=0}^{k}ess\sup_{M}|\na^{j}_{A}u|.
\end{equation}
\subsection{Chern-Simons Functional}
Let $t$ be the standard parameter on the factor $\mathbb{R}$ in the $Cyl(X):=\mathbb{R}\times X$, where $X$ is a closed, oriented $n$-dimensional Riemannian manifold that admits a smooth Riemannian metric $\rm{g}_{X}$, let $\{x^{j}\}_{j=1}^{n}$ be local coordinates of $X$. A connection $\textbf{A}$ over the cylinder $Cyl(X)$ is given by a local connection matrix $\textbf{A}=A_{0}dt+\sum_{i=1}^{n}A_{i}dx^{i}$, where $A_{0}$ and $A_{i}$ depend on all $n+1$ variable $t,x^{1},\ldots,x^{n}$. We take $A_{0}=0$ (sometimes called a temporal gauge) and denote $\textbf{A}=\sum_{i=1}^{n}A_{i}dx^{i}$. In this situation, the curvature is given by $F_{\textbf{A}}=F_{A}+dt\wedge\dot{A}$, where $\dot{A}=\frac{\pa A}{\pa t}$. We denote  by $\ast_{X}$ the Hodge star operator of $X$. If $\a$ is a $1$-form on $X$, then for Hodge star operator, $\ast$, defined on $Cyl(X)$ with respect to the product metric $dt^{2}+\rm{g}_{X}$, we have $\ast(dt\wedge\a)=\ast_{X}\a$.

We consider a cylinder over a closed manifold $X$ which admits smooth non-zero $3$-form $P$ and $4$-form $Q$. The instanton equation (\ref{E0}) is equivalent to
\begin{equation}\label{2.1}
\begin{split}
&\ast_{X}\dot{A}=-\ast_{X}P\wedge F_{A},\\ &\ast_{X}F_{A}=-\dot{A}\wedge\ast_{X}P-\ast_{X}Q\wedge F_{A}.\\
\end{split}
\end{equation}
Let $P$ be a $G$-bundle over $X$, the connection space $\mathcal{A}$ is an affine space modelled on $\Om^{1}(X,\mathfrak{g}_{P})$, so fixing a reference connection $A_{0}\in\mathcal{A}$, for any $A\in\mathcal{A}$, we can write $A=A_{0}+a$, $a\in\Om^{1}(X,\mathfrak{g}_{P})$. We define the Chern-Simons functional by
\begin{equation*}
CS(A):=-\int_{X}Tr\big{(}a\wedge{d_{A_{0}}a}+\frac{2}{3}a\wedge a\wedge a\big{)}\wedge\ast_{X}P,
\end{equation*}
fixing $CS(A_{0})=0$. This functional is obtained by integrating of the Chern-Simons $1$-form
\begin{equation*}
\Gamma_{A}(\b_{A})=-2\int_{X}Tr(F_{A}\wedge\b_{A})\wedge\ast_{X}P.
\end{equation*}
We find $CS$ explicitly by integrating $\Ga$ over paths $A(t)=A_{0}+ta$, from $A_{0}$ to any $A=A_{0}+a$:
\begin{equation*}
\begin{split}
CS(A)-CS(A_{0})&=\int_{0}^{1}\Gamma_{A(t)}\big{(}\dot{A}(t)\big{)}dt\\
&=-\int_{X}Tr\big{(}d_{A_{0}}a\wedge a+\frac{2}{3}a\wedge a\wedge a\big{)}\wedge\ast_{X}P+C,\\
\end{split}
\end{equation*}
where $C=C(A_{0},a)$ is a constant and vanishes if $A_{0}$ is an instanton. Suppose that the $3$-form $P$ is co-closed, i.e., $d\ast_{X}P=0$. The co-closed condition implies that the Chern-Simons $1$-form is closed. So it does not depend on the path $A(t)$ \cite{Earp2}.
\subsection{Non-degenerate flat connections}
We denote by 
\begin{equation*}
M(P,\rm{g}):=\{\Ga: F_{\Ga}=0\}/\mathcal{G}_{P},
\end{equation*}
the moduli space of gauge-equivalence class $[\Ga]$ of flat connection $\Ga$ on $P$. Following Uhlenbeck strong compactness theorem \cite{Uhlenbeck1982}, we know that the moduli space $M(P,\rm{g})$ is compact. We now begin to define the least eigenvalue of the self-adjoint operator $\De_{A}:=d_{A}d_{A}^{\ast}+d_{A}^{\ast}d_{A}$ with respect to connection $A$ on $L^{2}(X,\Om^{1}(\mathfrak{g}_{P}))$:
\begin{definition}\label{C7}
	The least eigenvalue of $\De_{A}$ on $L^{2}(X,\Om^{1}(\mathfrak{g}_{P}))$ is
	\begin{equation}
	\la(A):=\inf_{v\in\Om^{1}(\mathfrak{g}_{P})\backslash\{0\}}\frac{\langle\De_{A}v,v\rangle_{L^{2}}}{\|v\|^{2}}.
	\end{equation}
\end{definition}
In \cite{HT3} Lemma 3.3, the author showed that the function $\la[\cdot]$  with respect to the Uhlenbeck topology  is a continuous function on the moduli space of flat connections. We recall the definition of \textit{non-degenerate} flat connection, See \cite{Donaldson} Definition 2.4.
\begin{definition}\label{D1}
	A flat connection $\Ga$ on $P$ over a closed Riemannian manifold  $X$ is called \textit{non-degenerate}, if and
	only if $\ker{\De_{\Ga}}|_{\Om^{1}(\mathfrak{g}_{P})}=0$, i.e., $\la(A)>0$.
\end{definition}
Combining the compactness of moduli space of flat connections and the fact that the function $\la[\cdot]$ is continuous under Uhlenbeck topology, we then have 
\begin{proposition}\label{P1}(\cite{HT3} Proposition 3.6)
	If the flat connections over a closed Riemannian manifold $X$ are \textit{non-degenerate}, then there is a positive constant $\la=\la(X,\rm{g})$ such that
	\begin{equation*}
	\la(\Ga)\geq\la, \forall\ \Ga\in M(P,\rm{g}).
	\end{equation*}
\end{proposition}
For $A\in\mathcal{A}_{P}$ and $\delta>0$, we set
\begin{equation*}
T_{A,\delta}=\{a\in\Om^{1}(X,\mathfrak{g}_{P})\mid d^{\ast}_{A}a=0,\|a\|_{L^{\frac{n}{2}}_{1}(X)}\leq\de\}.
\end{equation*}
A neighbourhood of $[A]\in\mathfrak{B}:=\mathcal{A}_{P}/\mathcal{G}_{P}$ can be described as a quotient of $T_{A,\de}$, for small $\de$. In \cite{Taubes1990} Lemma 1.2, Taubes showed that the \textit{non-degenerate} flat connection (under moduli gauge transformation) is isolated on a closed three manifold. We will extend  this  property to higher dimensional  manifold.
\begin{proposition}\label{P3}
	Suppose that $A,\Ga$ are two distinct \textit{non-degenerate} flat connections over a closed Riemannian manifold  $X$. Then there is a positive constant $\de=\de(X,\rm{g})$ such that  
	\begin{equation*}
	\inf_{g\in\mathcal{G}}\|g^{\ast}(A)-\Ga\|_{L^{\frac{n}{2}}_{1}(X)}\geq\de.
	\end{equation*}
\end{proposition}
\begin{proof}
	We denote  by $A,\Ga$   two distinct  flat connections (under moduli gauge transformation). If there is a sufficiently small constant $\de>0$ such that
	\begin{equation*}
	\inf_{g\in\mathcal{G}}\|g^{\ast}(A)-\Ga\|_{L^{\frac{n}{2}}_{1}(X)}\leq\de,
	\end{equation*}  
	then we can choose a gauge transformation $g\in\mathcal{G}_{P}$ such that $g^{\ast}(A)$ and $\Ga$ satisfy relative gauge fixing, i.e.,
	\begin{equation*}
	d_{\Ga}^{\ast}(g^{\ast}(A)-\Ga)=0.
	\end{equation*}
	For simply, we also denote $g^{\ast}(A)$ to $A$. We denote $a:=A-\Ga$, therefore
	\begin{equation*}
	0=F_{A}=F_{\Ga}+d_{\Ga}a+a\wedge a=d_{\Ga}a+a\wedge a,
	\end{equation*} 
	Following Proposition \ref{P1}, it implies that,
	\begin{equation*}
	\la\|a\|^{2}_{L^{2}(X)}\leq\|d_{\Ga}a\|^{2}_{L^{2}(X)}+\|d^{\ast}_{\Ga}a\|^{2}_{L^{2}(X)}=\|d_{\Ga}a\|^{2}_{L^{2}(X)} .
	\end{equation*}
	Furthermore, the Weitzenb\"{o}ck formula gives
	\begin{equation*}
	(d_{\Ga}^{\ast}d_{\Ga}+d_{\Ga}d_{\Ga}^{\ast})a=\na_{\Ga}^{\ast}\na_{\Ga}a+Ric\circ a.
	\end{equation*}
	Combining the preceding identities yields,
	\begin{equation}\nonumber
	\begin{split}
	\|\na|a|\|^{2}_{L^{2}(X)}&\leq\|\na_{\Ga}a\|^{2}_{L^{2}(X)}\\
	&\leq\|d_{\Ga}a\|^{2}_{L^{2}(X)}+C\|a\|_{L^{2}(X)}^{2}\\
	&\leq (1+C\la^{-1})\|d_{\Ga}a\|_{L^{2}(X)}^{2}\\
	&\leq (1+C\la^{-1})\|a\wedge a\|_{L^{2}(X)}^{2}\\
	&\leq (1+C\la^{-1})\|a\|^{2}_{L^{n}(X)}\|a\|^{2}_{L^{\frac{2n}{n-2}}(X)}\\
	&\leq C(1+\la^{-1})\|a\|^{2}_{L^{\frac{n}{2}}_{1}(X)}\|a\|^{2}_{L^{2}_{1}(X)}\\
	&\leq C(1+\la^{-1})^{2}\|\na|a|\|^{2}_{L^{2}(X)}\|a\|^{2}_{L^{\frac{n}{2}}_{1}(X)},\\
	\end{split}
	\end{equation}
	where we use the Sobolev embedding $L^{\frac{n}{2}}_{1}\hookrightarrow L^{n}$, $L^{2}_{1}\hookrightarrow L^{\frac{2n}{n-2}}$ and Kato inequality $|\na|a||\leq|\na_{\Ga}a|$, $C=C(X,\rm{g})$ is a positive constant. If we choose $\|a\|_{L^{\frac{n}{2}}_{1}(X)}$ sufficiently small to ensure that $\|a\|^{2}_{L^{\frac{n}{2}}_{1}(X)}\leq\frac{\la^{2}}{2C(1+\la)^{2}}$, then
	$a\equiv0$. It contradicts to our initial assumption regarding the connections $A,\Ga$. 
\end{proof}
For any flat connection $[\Ga]$ over a closed $G_{2}$- or Calabi-Yau manifold $X$ with full holonomy, i.e., $\pi_{1}(X)$ is finite, we can show that the least eigenvalue of $\De_{\Ga}$ is positive, i.e., the flat connections are all \textit{non-degenerate}. One also can see \cite{HT2} Theorem  1.1.
\begin{proposition}\label{P2}
	Let $G$ be a compact Lie group, $P$ be a $G$-bundle over a closed, smooth $G_{2}$- or Calabi-Yau manfold $X$. If $X$ has full holonomy, then the flat connections on $X$  are  \textit{non-degenerate}.
\end{proposition}
\begin{proof}
	We denote $\Ga$ by a flat connection on $X$. For any harmonic $1$-form $\a$ with respect to $\De_{\Ga}:=d_{\Ga}^{\ast}d_{\Ga}+d_{\Ga}d_{\Ga}^{\ast}$, following the Weitzenb\"{o}ck formula, we have $\na_{\Ga}\a=0$. Here we use the vanishing of the Ricci curvature on $G_{2}$- or Calabi-Yau manifolds and $\Ga$ is flat.
	
	Let $R_{ij}dx^{i}\wedge dx^{j}$ denote the Riemann curvature tensor viewed as an $ad(T^{\ast}X)$ valued $2$-form,\ The vanishing of $\na_{\Ga}\a$ implies
	\begin{equation*}
	0=[\na_{i},\na_{j}]\a=ad((F_{ij})+R_{ij})\a,
	\end{equation*}
	for all $i,j$. Since $F_{ij}$ vanishes, $R_{ij}\a=0$, and the components of $\a$ are in the kernel of the Riemannian curvature operator. This reduces the Riemannian holonomy group, unless $\a=0$ which implies $\ker{\De_{\Ga}}|_{\Om^{1}(X,\mathfrak{g}_{P})}=0$. Thus, we have the dichotomy: $\a\neq0$ implies a reduction of the holonomy of $X$, and $\a=0$ implies the connection $A$ is \textit{non-degenerate}.
\end{proof}

\subsection{A priori estimate for Yang-Mills equation}
We will recall the monotonicity formula for Yang-Mills equation \cite{Price,Tian}. Let $M$ be a compact Riemannian $n$-manifold with a smooth Riemannian metric $\rm{g}$ and $E$ be a smooth vector bundle over $M$ with compact structure group $G$. Let $p\in M$, let $r_{p}\leq inj(M)$ be a positive number with properties: there are normal coordinates $x_{1},\cdots,x_{n}$ in the geodesic $B_{r_{p}}(p)$ of $(M,\rm{g})$, such that $p=(0,\cdots,0)$ and for some constant $c(p)$:
\begin{equation*}
|\rm{g}_{ij}-\de_{ij}|\leq c(p)r^{2},\ |d\rm{g}_{ij}|\leq c(p)r,
\end{equation*}
where $\rm{g}_{ij}=\rm{g}\langle\frac{\pa}{\pa x_{i}},\frac{\pa}{\pa x_{j}}\rangle$.
\begin{remark}\label{R1}
	The constant $r_{p}$ and $c(p)$ can be chosen depending only on the injective radius at $p$ and the curvature of $\rm{g}$. Suppose that the manfold $M$ is compact, then the constant $r_{p}$ has a lower  positive bounded constant $r_{M}$  and $|c(p)|$ has a upper positive constant $c_{M}$. The constants $r_{M}$ and $c_{M}$  depend only  on $M$ and $\rm{g}$. 
\end{remark}
We will always denote $O(1)$ a quantity bounded by a constant depending only on $n$. For any Yang-Mills connection $A$ of $E$, we have 
\begin{theorem}\label{2.12}
	Let  $A$ be any  Yang-Mills connection of a $G$-bundle  over a compact manifold $M$ with smooth Riemannian metric $\rm{g}$. For any $p\in M$, there are positive constants  $r_{p}, c(p)$ and $a_{0}$ are  positive constant depend on $M,\rm{g}$ such that for any $0<\sigma<\rho<r_{p}$ and $a\geq a_{0}$, we have
	\begin{equation}\nonumber
	\begin{split}
	&\quad\rho^{4-n}e^{a\rho^{2}}\int_{B_{\rho}(p)}|F_{A}|^{2}dvol_{\rm{g}}-\sigma^{4-n}e^{a\sigma^{2}}\int_{B_{\sigma}(p)}|F_{A}|^{2}dvol_{\rm{g}}\\
	&\geq4\int_{B_{\rho}(p)\backslash B_{\sigma}(p)}r^{4-n}e^{ar^{2}}|\frac{\pa}{\pa r}\lrcorner F_{A}|^{2}dvol_{\rm{g}}+\int_{\sigma}^{\rho}\big{(}e^{a\tau^{2}}\tau^{4-n}(2a\tau-O(1)c(p)\tau)\int_{B_{\tau}(p)}|F_{A}|^{2}dvol_{\rm{g}}\big{)}d\tau,\\
	\end{split}
	\end{equation}
\end{theorem}
\begin{proof}
	We choose, for any $\tau$ small enough, $\xi(r)=\xi_{\tau}(r)=\eta(\frac{r}{\tau})$, where $\eta$ is smooth and satisfies: $\eta(r)=1$ for $r\in[0,1]$, $\eta(r)=0$ for $r\in[1+\varepsilon,\infty)$, $\varepsilon>0$ and $\eta'(r)\leq 0$. By taking $\phi=1$ on \cite{Tian} Equation (2.1.8), we have
	\begin{equation}\nonumber
	\begin{split}
	&\quad\frac{\pa}{\pa \tau}(\tau^{4-n}e^{a\tau^{2}}\int_{M}\xi_{\tau}|F_{A}|^{2}dV_{g})\\
	&=4\tau^{4-n}e^{a\tau^{2}}\big{(}\frac{\pa}{\pa\tau}(\int_{M}\xi_{\tau}|\frac{\pa}{\pa r}\lrcorner F_{A}|^{2}dV_{g})+(-O(1)c(p)+2a)\tau\int_{M}\xi_{\tau}|F_{A}|^{2}dV_{g}\big{)}.\\
	\end{split}
	\end{equation}
	Then, by integrating on $\tau$ and letting $\varepsilon$ tends to zero, we complete the proof of this theorem.
\end{proof}
We then recall the $\varepsilon$-regular theorem of Yang-Mills equation \cite{Tian,Uhlenbeck1982}.
\begin{theorem}\label{2.19}(\cite{Tian} Theorem 2.2.1)
	Let $A$ be any Yang-Mills connection of a $G$-bundle over a compact manifold $M$ with smooth Riemannian metric $\rm{g}$. Then there exist positive constants $\varepsilon=\varepsilon(M,\rm{g},n)$ and $C=C(M,\rm{g},n)$ which depend on $M,\rm{g},n$, such that for any $p\in M$ and $0<\rho<r_{p}$, whenever 
	\begin{equation*}
	\rho^{4-n}\int_{B_{\rho}(p)}|F_{A}|^{2}dvol_{\rm{g}}\leq\varepsilon,
	\end{equation*}
	then
	\begin{equation*}
	|F_{A}|(p)\leq\frac{C}{\rho^{2}}\big{(}\rho^{4-n}\int_{B_{\rho}(p)}|F_{A}|^{2}dvol_{\rm{g}}\big{)}^{\frac{1}{2}}.
	\end{equation*}
\end{theorem}
The constant $r_{p}$ in Theorem \ref{2.12} and Theorem \ref{2.19} is the same.  We then have a useful $L^{\infty}$ estimate for Yang-Mills connection $A$ when the $L^{2}$-norm of curvature $F_{A}$ is sufficiently small. 
\begin{corollary}\label{C1}
	Let $A$ be any Yang-Mills connection of a principal $G$-bundle over a compact manifold $M$ with smooth Riemannian metric $\rm{g}$. Then there are  positive constants $\varepsilon_{0}=\varepsilon_{0}(M,\rm{g},n)$ and $C=C(M,\rm{g},n)$ with following significance. If the curvature $F_{A}$ of connection $A$ obeys
	\begin{equation*}
	\|F_{A}\|_{L^{2}(M)}\leq\varepsilon_{0},
	\end{equation*} 
 then
	\begin{equation}\label{E2}
	\|F_{A}\|_{L^{\infty}(M)}\leq C\|F_{A}\|_{L^{2}(M)}.
	\end{equation} 
\end{corollary}
\begin{proof}
Since the Theorem \ref{2.12} holds for all  sufficiently large constant $a$,  we can choose the constant $a$ to ensure that
	\begin{equation*}
	2a-O(1)c_{M}>0.
	\end{equation*}
where $O(1)$ is the constant appear in Theorem \ref{2.12}, $c_{M}$ is a upper positive constant of $|c(p)|$. Hence for any $0<\sigma<\rho< r_{M}$, 
	$$\sigma^{4-n}\int_{B_{\sigma}(p)}|F_{A}|^{2}dvol_{\rm{g}}\leq \rho^{4-n}e^{a(\rho^{2}-\sigma^{2})}\int_{B_{\rho}(p)}|F_{A}|^{2}dvol_{\rm{g}}\leq \rho^{4-n}e^{a\rho^{2}}\int_{M}|F_{A}|^{2}dvol_{\rm{g}},$$
	We now let $\rho\nearrow r_{M}$, hence 
	\begin{equation}\label{E1}
	\sigma^{4-n}\int_{B_{\sigma}(p)}|F_{A}|^{2}dvol_{\rm{g}}\leq C\int_{M}|F_{A}|^{2}dvol_{\rm{g}},\ \forall \sigma\in(0,r_{M}),
	\end{equation}
	where $C$ is a positive constant depending on $M$ and metric $\rm{g}$. 
	
	For any $p\in M$, in the geodesic ball $B_{\sigma}(p)$, $0<\sigma<r_{M}$, following Equation (\ref{E1}), we have
	\begin{equation*}
	\sigma^{4-n}\int_{B_{\sigma}(p)}|F_{A}|^{2}dvol_{\rm{g}}\leq C\varepsilon_{0},\ \forall \sigma\in(0,r_{M}),
	\end{equation*}
	where $C=C(M,\rm{g})$ is a positive constant. We choose  $\varepsilon_{0}$ small enough to ensure that $C\varepsilon_{0}\leq\varepsilon$, where $\varepsilon$ is the constant in Theorem \ref{2.19}. Hence following Theorem \ref{2.19},
	\begin{equation*}
	|F_{A}|(p)\leq \frac{C}{\sigma^{2}}(\sigma^{4-n}\int_{B_{\sigma}(p)}|F_{A}|^{2}dvol_{\rm{g}})^{\frac{1}{2}}\leq \sigma^{-\frac{n}{2}}C\|F_{A}\|_{L^{2}(X)},\ \forall\sigma\in(0,r_{M}).
	\end{equation*}
	We now let $\sigma\nearrow r_{M}$, thus 
	\begin{equation*}
	|F_{A}|(p)\leq r_{M}^{-\frac{n}{2}}C\|F_{A}\|_{L^{2}(X)},\  \forall p\in M.
	\end{equation*}
	We complete this proof.
\end{proof}
\section{Asymptotic Behavior}

\begin{definition}\label{D2}
	Let $X$ be a closed, smooth manifold of dimension $n\geq 4$ with a Riemannian metric $\rm{g}_{X}$. We call $X$  a \textbf{good} manifold, if $X$ admits non-zero, smooth $3$-form $P$ and $4$-form $Q$ satisfying $d\ast_{X}P=d\ast_{X}Q=0$. 
\end{definition}
\begin{example}
	There are many manifolds are \textbf{good} in the sense of Definition \ref{D2}. For example:\\
	(1) $X$ is a Calabi-Yau 3-fold. It is defined as a manifold with a K\"{a}hler $(1,1)$-form $\w$ and a holomorphic form $\Om\in\Om^{3,0}$. We can construct a $G_{2}$-structure on $Cyl(X)$:
	\begin{equation*}
	\phi=dt\wedge\w+Im\Om.
	\end{equation*}
	We denote $(P,Q):=(Re\Om,\frac{1}{2}\w^{2})$, one can see the instanton equation (\ref{E0}) is a $G_{2}$-instanton.\\
	(2) $X$ is a parallel $G_{2}$-manifold. It is defined as a manifold with a $G_{2}$-structure $3$-from $\phi$. We can construct a $Spin(7)$-structure on $Cyl(X)$:
	\begin{equation*}
	\Phi=dt\wedge\phi+\ast_{X}\phi.
	\end{equation*}
	We denote  $(P,Q):=(\phi,\ast\phi)$, one also can see the instanton equation (\ref{E0}) is a $Spin(7)$-instanton.
\end{example}
Taking the exterior derivative of (\ref{E0}), then  using the Bianchi identity and the fact $(P,Q)$ is co-closed, it's easy to see the solution of instanton equation (\ref{E0}) also satisfies Yang-Mills equation. In this section, we denote by $X$  a closed \textbf{good} manifold. We begin to study the decay of instantons over tubular ends. At first, we consider a family of bands $B_{T}=[T,T+1]\times X$ which we identify with the model $B_{0}=[0,1]\times X$ by translation. So the integrability of $|F_{\mathbf{A}}|^{2}$ over the end implies that $\int_{[T,T+1]\times X}|F_{\mathbf{A}}|^{2}\rightarrow0$ as $T\rightarrow\infty$. On the compact manifold $B_{T}$, we can choose a positive constant $r_{T}$ by the similar way in above Section 2.3, See Remark \ref{R1}. Since  $B_{T}$  is identified with $B_{0}$ by translation, it's easy to see $r_{T}$ is not dependent on $T$.  We then have
\begin{proposition}\label{3.4}
	Suppose that $\mathbf{A}$ is an instanton on a cylinder $Cyl(X)$ over a closed \textbf{good} manifold  with $L^{2}$-curvature $F_{\mathbf{A}}$. Suppose also that the flat connections on $X$ are non-degenerate. Then at the end of $Cyl(X)$, there is a flat connection $\Gamma$ over $X$ such that $\mathbf{A}$ converges to $\Gamma$, i.e. the restriction $\mathbf{A}|_{X\times\{T\}}$ converges (modulo gauge equivalence) to $\Ga$ in $C^{\infty}$ over $X$ as $T\rightarrow\infty$.
\end{proposition}
\begin{proof}
The argument is similar to Proposition 4.1 on \cite{Donaldson}. Let $A(T)$ be the connection over $B_{T}$, so the integrability of $|F_{\textbf{A}}|^{2}$ and Corollary  \ref{C1} imply that  $\|F_{A(T)}\|_{L^{\infty}(X)}\rightarrow0$ as $T\rightarrow\infty$.	
Uhlenbeck's weak compactness theorem implies that for any sequence $T_{i}\rightarrow\infty$ there are subsequence $T'_{i}$ and a flat connection $\Ga$ such that, after suitable gauge transformations $A(T'_{i})\rightarrow\Ga$, in $C^{\infty}$ over compact subsets of $B_{0}$. In particular the restriction of $A(T'_{i})$ to the cross-section $X\times\{1/2\}$ converges in $C^{\infty}$ to $\Ga$.
	
We will also claim that  the  limit is unique. First consider the metric 
$$dist_{L^{2}_{1}}([A],[B])=\inf_{g\in\mathcal{G}}\|A-g^{\ast}(B)\|_{L^{2}_{1}(X)} $$
on the space of equivalence classes $\mathcal{B}_{X}$ of connections over $X$. We denote by $[A(T)]$ the equivalence class of $A(T)$. Since the  path $A(T)$ is  continuous and 
$$dist_{L^{2}_{1}(X)}([A(T_{1})],[A(T_{2})])\leq \|A(T_{1})-A(T_{2})\|_{L^{2}_{1}(X)},$$
the path $[A(T)]$ is also continuous on $\mathcal{B}_{X}$ under $L_{1}^{2}$-topology. We have shown that the continuous path $A(T)$ converges (under moduli transformation) to the space of flat connections $M(P,g)$ in the sense that $dist_{L^{2}_{1}}([A_{T}],M(P,g))\rightarrow0$ over $X$ as $T\rightarrow\infty$. But the points of $M(P,g)$ are isolated so there is a $\de>0$ with $\inf_{g\in\mathcal{G}}\|g^{\ast}(\Ga)-\tilde{\Ga}\|_{L^{\frac{n}{2}}_{1}(X)}\geq\de$, See Proposition \ref{P3} for distinct points $\Ga$, $\tilde{\Ga}$ in $M(P,g)$. For large $T$ the connections $A(T)$ have distance less than $\de/2$ form some point of $M(P,g)$. But by the intermediate value theorem applied to the continuous function $\inf_{g\in\mathcal{G}}\|g^{\ast}(A)-\Ga\|_{L^{\frac{n}{2}}_{1}(X)}$  this point must be independent of $T$, so we obtain a flat connection $\Ga$ such that $A(T)$ converges to $\Ga$ in the $L^{\frac{n}{2}}_{1}$-distance. But now it follows similarly that the convergence is $C^{\infty}$, since any sequence has a $C^{\infty}$-convergent subsequence. 
\end{proof}

By Proposition \ref{3.4}, the moduli space $\M$ of finite energy instantons on $Z$ is the disjoint union of its subsets $\M(A_{-\infty},A_{+\infty})$,\ where $A_{-\infty}$ and $A_{+\infty}$ run over all components of the space of flat connection on $X$ and $\M(A_{-\infty},A_{+\infty})$ is the subset of $\M$ consisting of instantons with limits in $A_{-\infty}$ and in $A_{+\infty}$ over the two ends respectively.  We apply a key result due to Uhlenbeck for the connections with $L^{p}$-small curvature (\cite{Uhlenbeck1985} Corollary 4.3) to prove that  
\begin{corollary}\label{C4.5}
Assume the hypothesis in Proposition \ref{3.4}. Suppose also that the flat connections on $X$ are \textit{non-degenerate}. Then there exists positive constants $C,T$  with following significance. If $t>T$, there exist a family gauge transformation $g(t)$ such that
	\begin{equation*}
	\|g(t)^{\ast}(A(t))-\Ga\|_{L^{p}_{1}(X)}\leq C\|F_{A(t)}\|_{L^{p}(X)},\ \forall\ 2p>n.
	\end{equation*}
\end{corollary}
\begin{proof}
	We denote $B_{t}=[t,t+1]\times X$ and $2p>n$. For the constant $\varepsilon_{0}$ which satisfies the hypothesis in Corollary \ref{C1}, we can choose a large enough $T$ to ensure that 
	\begin{equation*}
	\|F_{\mathbf{A}}\|_{L^{2}(B_{t})}\leq \varepsilon_{0},\ \forall t>T.
	\end{equation*}
	Therefore following the estimate in Corollary \ref{C1}, we have
	\begin{equation}\nonumber
	\|F_{A(t)}\|_{L^{p}(X)}\leq({\rm{Vol}(X)})^{\frac{1}{p}}\|F_{A(t)}\|_{L^{\infty}(X)}\leq({\rm{Vol}(X)})^{\frac{1}{p}}\|F_{\mathbf{A}}\|_{L^{\infty}(B_{t})}\leq C\|F_{\mathbf{A}}\|_{L^{2}(B_{t})},
	\end{equation}
	where $C=C(X,{\rm{g}},p)$ is a positive constant. We can choose $C\|F_{\mathbf{A}}\|_{L^{2}(B_{t})}\leq\varepsilon$,
	where constant $\varepsilon$ satisfies the hypothesis in \cite{Uhlenbeck1985} Corollary 4.3. Then there exist a flat connection $\Ga(t)$  and a gauge transformation $g(t)$ such that
	\begin{equation}\nonumber
	\|g(t)^{\ast}(A(t))-\Ga(t)\|_{L^{q}_{1}(X)}\leq C(p)\|F_{A(t)}\|_{L^{q}(X)},\ \forall\ 2q>n+1.
	\end{equation}
Combining the preceding inequalities  yields
	$$\|g^{\ast}(A(t))-\Ga(t)\|_{L^{\frac{n}{2}}_{1}(X)}\leq C\|g^{\ast}(A(t))-\Ga(t)\|_{L^{q}_{1}(X)}\leq\|F_{A(t)}\|_{L^{q}(X)}\leq C\|F_{\textbf{A}}\|_{L^{2}(B_{t})}.$$
	Since the flat connections over $X$ are isolated, the limit is unique--independent of the sequence and subsequence chosen, then
	$A(t)\rightarrow\Ga$ in $C^{\infty}$ under moduli gauge transformation.  Hence for any small enough positive constant $\varepsilon_{0}$, there is a large enough constant $t$ such that,
	\begin{equation*}
	\inf_{g\in\mathcal{G}}\|g^{\ast}(A(t))-\Ga\|_{L^{\frac{n}{2}}_{1}(X)}\leq\varepsilon_{0}.
	\end{equation*}
	Therefore 
	\begin{equation}\nonumber 
	\begin{split}
	\inf_{g\in\mathcal{G}}\|\Ga(t)-\Ga\|_{L^{\frac{n}{2}}_{1}(X)}&\leq \inf_{g\in\mathcal{G}}\|g^{\ast}(A(t))-\Ga\|_{L^{\frac{n}{2}}_{1}(X)}+\inf_{g\in\mathcal{G}}\|g^{\ast}(A(t))-\Ga(t)\|_{L^{\frac{n}{2}}_{1}(X)}\\
	&\leq\varepsilon_{0}+C\|F_{\textbf{A}}\|_{L^{2}(B_{t})},\\
	\end{split}
	\end{equation}
where $C=C(X,{\rm{g}},n)$ is a positive constant. We can choose $T$ sufficiently large to ensure that $\varepsilon+C\|F_{\textbf{A}}\|_{L^{2}(B_{t})}<\de$, where $\de$ satisfies the hypothesis in Proposition \ref{P3}. Thus $[\Ga(t)]=[\Ga]$. We complete the proof of this corollary.
\end{proof}
We begin to study the decay of instantons over the ends. At first, we prove a useful lemma as follows.
\begin{lemma}\label{L2}
	Suppose that $\mathbf{A}$ is an instanton with $L^{2}$-curvature $F_{\textbf{A}}$ over $Cyl(X)$, where $X$ is a closed \textbf{good} manifold. Suppose also that the flat connections on $X$  are \textit{non-degenerate}. We then have
	\begin{equation}\label{E5}
	CS(A(T))-CS(A_{\infty})=-\int_{[T,\infty)\times{X}}Tr(F_{\mathbf{A}}\wedge\ast{F}_{\mathbf{A}}).
	\end{equation}
\end{lemma}
\begin{proof}
	Following the definition of $CS$, we then have
	\begin{equation}\nonumber
	\begin{split}
	&\quad	CS(A(T'))-CS(A(T))=-2\int_{[T,T']\times X}Tr(F_{A(t)}\wedge dt\wedge\dot{A}(t))\wedge\ast_{X}P\\
	&=-\int_{[T,T']\times X}Tr(F_{\mathbf{A}}\wedge F_{\mathbf{A}})\wedge\ast \Om+\int_{T}^{T'}\Big{(}\int_{X}Tr(F_{A(t)}\wedge F_{A(t)})\wedge\ast_{X}Q\Big{)}dt\\
	\end{split}
	\end{equation}
	Following Proposition \ref{3.4}, there exist a flat connection $\Ga$ over $X$ such that, $\mathbf{A}|_{X\times \{T_{i}\}}$ converges to $\Ga$ in $C^{\infty}$ after suitable gauge transformations. Since $Q$ is co-closed, following Chern-Weil theory, we have 
	\begin{equation*}
	\begin{split}
	\int_{X}Tr(F_{A(t)}\wedge F_{A(t)})\wedge\ast_{X}Q&=\lim_{t\rightarrow\infty}\int_{X}Tr(F_{A(t)}\wedge F_{A(t)})\wedge\ast_{X}Q\\
	&=\int_{X}Tr(F_{\Ga}\wedge F_{\Ga})\wedge\ast_{X}Q=0.\\
	\end{split}
	\end{equation*}
	Taking the limit over finite tubes $(T,T')\times X$ with $T'\rightarrow +\infty$ and $\mathbf{A}$ satisfies the instanton equation, we prove the identity (\ref{E5}). 
\end{proof}
\begin{theorem}\label{T1}
	Suppose $\mathbf{A}$ is a smooth solution of instanton (\ref{E0}) with $L^{2}$-curvature $F_{\mathbf{A}}$. Suppose also that the flat connections on $X$ are \textit{non-degenerate}. Then there are positive constants $C',C''$ such that
	\begin{equation*}
	|F_{\mathbf{A}}|\leq C''e^{-C'|t|},
	\end{equation*}
	for sufficiently large $|t|$.
\end{theorem}
\begin{proof}
	We prove this using a differential inequality derived from the instanton on $Cyl(X)$. Our proof here is similar to Donaldson's arguments in \cite{Donaldson} Section 4.2 for ASD connection. For $T>0$, we set
	\begin{equation*}
	J(T)=\int_{T}^{\infty}\|F_{\mathbf{A}}\|_{L^{2}(X)}^{2}=-\int_{[T,\infty)\times X}Tr(F_{\mathbf{A}}\wedge\ast F_{\mathbf{A}})=\int_{[T,\infty)\times X}Tr(F_{\mathbf{A}}\wedge F_{\mathbf{A}})\wedge\ast\Om.
	\end{equation*}
	Following Lemma \ref{L2}, we have
	\begin{equation}\label{3.5}
	J(T)=CS(A(T))-CS(A_{\infty})
	\end{equation}
	where $A(T)$ is the connection over $X$ obtain by restriction to $X\times\{T\}$. Following (\ref{3.5}), we obtain the $T$ derivative of $J$ as
	\begin{equation}\label{3.2}
	\frac{d}{dT}J(T)=\frac{d}{dT}\big{(}CS(A(T))-CS(A_{\infty})\big{)}.
	\end{equation}
	On the other hand, the $T$ derivative of $J(T)$ can be expressed as minus the integration over $X\times\{T\}$ of the curvature density $|F_{\mathbf{A}}|^{2}$, and this is exactly the $n$-dimensional curvature density $|F_{A(T)}|^{2}$ plus the density $|\dot{A}|^{2}$. By the relation Equation (\ref{2.1}) between the two components of the curvature for an instanton, we have
	\begin{equation*}
	\begin{split}
	\|F_{A(T)}\|^{2}_{L^{2}(X)}&=\int_{X}Tr(F_{A}\wedge\dot{A}\wedge\ast_{X}P)+\int_{X}Tr(F_{A}\wedge F_{A})\wedge\ast_{X}Q\\
	&=\|\dot{A}(T)\|_{L^{2}(X)}^{2}+\int_{X}Tr(F_{A}\wedge F_{A})\wedge\ast_{X}Q.\\
		\end{split}
	 \end{equation*}
   Following Chern-Weil theory, we have $\int_{X}Tr(F_{A}\wedge F_{A})\wedge\ast_{X}Q=0$, one also can see the proof of Lemma \ref{L2}. Hence  $\|F_{A(T)}\|^{2}_{L^{2}(X)}=\|\dot{A}(T)\|_{L^{2}(X)}^{2}$. Thus
	\begin{equation}\label{3.3}
	\frac{d}{dT}J(T)=-2\|F_{A(T)}\|^{2}_{L^{2}(X)}
	\end{equation}
	From (\ref{3.2}) and (\ref{3.3}), we have
	\begin{equation*}
	\frac{d}{dT}\big{(}CS(A(T))-CS(A_{\infty})\big{)}=-2\|F_{A(T)}\|^{2}_{L^{2}(X)}
	\end{equation*}
	To connect these two observations we establish an inequality between the Chern-Simon function $CS(A(T))$ and $\|F_{A(T)}\|_{L^{2}(X)}$, valid for any connection over $X$ which is close to $A_{\infty}$.\ We write, for fixed large $T$,
	\begin{equation*}
	A(T)=A_{\infty}+a,
	\end{equation*}
	where $A_{\infty}$ is a flat connection over $X$, so we may suppose that $a$ is small as we please in $C^{\infty}$. Also,we may suppose that $a$ satisfies the Coulomb gauge fixing:
	\begin{equation*}
	d_{A_{\infty}}^{\ast}a=0.
	\end{equation*}
	Now, we have
	\begin{equation}\nonumber
	\begin{split}
	CS(A(T))-CS(A_{\infty})&=-\int_{X}Tr(d_{A_{\infty}}a\wedge a+\frac{2}{3}a\wedge a\wedge a)\wedge\ast_{X}P\\
	&=-\int_{X}Tr(\frac{1}{3}d_{A_{\infty}}a\wedge a+\frac{2}{3}F_{A(T)}\wedge a)\wedge\ast_{X}P.\\
	\end{split}
	\end{equation}
	We use the fact that the kernel of $d_{A_{\infty}}+d_{A_{\infty}}^{\ast}$ in $\Om^{1}$ is trivial, so following Proposition \ref{P1} there exist a positive constant $\la$ such that
	\begin{equation*}
	\|a\|_{L^{2}(X)}\leq\la\|d_{A_{\infty}}a\|_{L^{2}(X)}.
	\end{equation*}
	We observe that
	\begin{equation}\nonumber
	\begin{split}
	|\int_{X}Tr(d_{A_{\infty}}a\wedge a)\wedge\ast_{X}P|&\leq\|a\|_{L^{2}(X)}\|d_{A_{\infty}}a\|_{L^{2}(X)}\max_{X}|P|\\
	&\leq C\la\|d_{A_{\infty}}a\|^{2}_{L^{2}(X)},
	\end{split}
	\end{equation}
	and
	\begin{equation}\nonumber
	\begin{split}
	|\int_{X}Tr(F_{A(T)}\wedge a)\wedge\ast_{X}Q|&\leq\|F_{A(T)}\|_{L^{2}(X)}\|a\|_{L^{2}(X)}\max_{X}|Q|\\
	&\leq C\la\|d_{A_{\infty}}a\|_{L^{2}(X)}\|F_{A(T)}\|_{L^{2}(X)},\\
	\end{split}
	\end{equation}
	where $C=C(X,P,Q)$ is a positive constant.Hence, we get
	\begin{equation}\label{E3.5}
	CS(A(T))-CS(A_{\infty})\leq C\la(\|d_{A_{\infty}}a\|^{2}_{L^{2}(X)}+\|d_{A_{\infty}}a\|_{L^{2}(X)}\|F_{A(T)}\|_{L^{2}(X)}).
	\end{equation}
	We denote $B_{t}:=[t,t+1]\times X$. Let $p\geq n$ be a positive constant. Following Corollary \ref{C4.5} and Sobolev embedding $L^{p}_{1}\hookrightarrow L^{\infty}$, for large enough $T$, we obtain that
	\begin{equation}\label{E4.8}
	\|a(t)\|_{L^{\infty}(X)}\leq C\|a(t)\|_{L^{p}_{1}(X)}\leq C\|F_{A(t)}\|_{L^{p}(X)}\leq C\|F_{\mathbf{A}}\|_{L^{p}(B_{t})}\leq C\|F_{\mathbf{A}}\|_{L^{2}(B_{t})}.\\
	\end{equation}
	where $C=C(X,{\rm{g}},p)$ is a positive constant. On the other hand $F_{A(T)}=F_{A_{\infty}+a}=d_{A_{\infty}}a+a\wedge a$. Therefore
	\begin{equation}\label{E3.7}
	\|F_{A(T)}\|_{L^{2}(X)}\geq\|d_{A_{\infty}}a\|_{L^{2}(X)}-\|a\wedge a\|_{L^{2}(X)}\geq\|d_{A_{\infty}}a\|_{L^{2}(X)}-\|a\|_{L^{\infty}(X)}\|a\|_{L^{2}(X)}.\\
	\end{equation}
	Combining the preceding inequalities (\ref{E4.8}) and (\ref{E3.7}) yields,
	\begin{equation}\nonumber
	\|F_{A(T)}\|_{L^{2}(X)}\geq\|d_{A_{\infty}}a\|_{L^{2}(X)}-C\la\|F_{\mathbf{A}}\|_{L^{2}(B_{t})}\|d_{A_{\infty}}a\|_{L^{2}(X)}.\\
	\end{equation}
	Provided $C\la\|F_{\mathbf{A}}\|_{L^{2}(B_{t})}\leq1/2$, rearrangement gives
	\begin{equation}\label{E3.8}
	\|F_{A(T)}\|_{L^{2}(X)}\geq 1/2\|d_{A_{\infty}}a\|_{L^{2}(X)}.
	\end{equation}
	Combining the preceding inequalities (\ref{E3.5}) and (\ref{E3.8}) yields,
	\begin{equation}\nonumber
	CS(A(T))-CS(A_{\infty})\leq C\|F_{A(T)}\|^{2}_{L^{2}(X)}.
	\end{equation}
	Putting all these together, we get a differential inequality, when $T$ is large enough,
	\begin{equation*}
	J(T)\leq-C\frac{d }{dT}J(T).
	\end{equation*}
	It is easy to see that this implies that $J$ decays exponentially,
	\begin{equation*}
	J(T)\leq C''e^{-C'T},
	\end{equation*}
	where $C'=(C)^{-1}$ and $C''=J(T_{0})e^{-C'T_{0}}$. For the constant $\varepsilon_{0}$ which satisfies the hypothesis in Corollary \ref{C1}, we can choose a large enough $T$ to ensure that  $$\|F_{\mathbf{A}}\|_{L^{2}(B_{T})}\leq J(T)\leq C''e^{-C'T}\leq \varepsilon_{0},$$ 
	where $B_{T}=[T,T+1]\times X$. Finally, following the estimate in Corollary \ref{C1}, it implies that $|F_{\mathbf{A}}|$ has  exponential decay.
\end{proof}
We observe that the flat connections on a $G$-bundle over a closed Calabi-Yau $3$-fold or $G_{2}$-manifold with full holonomy, i.e., the manifold has finite fundamental group, are all  \textit{non-degenerate}. As a simply application, we have
\begin{corollary}
	Let $X$ be a closed Calabi-Yau $3$-fold (or $G_{2}$-manifold) with full holonomy, $P$ be a principal $G$-bundle over cylinder $Cyl(X)$ with $G$ being a compact Lie group. If $\mathbf{A}$ is a $G_{2}$- (or $Spin(7)$-) instanton with $L^{2}$-curvature $F_{\mathbf{A}}$, then there exist positive constants $C',C''$ such that
	\begin{equation*}
	|F_{\mathbf{A}}|\leq C''e^{-C'|t|},
	\end{equation*}
	for sufficiently large $|t|$.	
\end{corollary}
\section{Real Killing  spinor manifold}
Let $(X,\rm{g})$ be a real Killing spinor compact manifold of dimension $n$, i.e., there are non-zero $3$-form $P$ and $4$-form $Q$ which satisfy
\begin{equation*}
dP=4Q,\ d\ast_{X}Q=(n-3)\ast_{X}P,
\end{equation*}
where $\ast_{X}$ is the Hodge star operator on $X$. For $n>3$, the Chern-Simons functional can then be written as
\begin{equation}\label{E4.6}
CS(A)=-\frac{1}{2(n-3)}\int_{X}Tr(F_{A}\wedge F_{A})\wedge\ast_{X}Q,
\end{equation}
which is gauge-invariant. We consider the cylinder $Cyl(X)$ over $X$. The instanton equation (\ref{E0}) on the cylinder splits into the two equations (\ref{2.1}). The gradient flow equation for $CS(A)$ is then the first equation (\ref{2.1}). The gradient flow of Chern-Simons functional (\ref{E4.6}) may not be equivalent to  the instanton equation (\ref{2.1}). But if $X$ is a nearly parallel $G_{2}$- or nearly K\"{a}hler manifold, the first equation on (\ref{2.1}) implies the second equation. So on a nearly parallel $G_{2}$- or nearly K\"{a}hler  manifold, the instanton equation on the cylinder is equivalent to the gradient flow for $CS(A)$, See \cite{HN} Section 4.4 and \cite{HILP}.
\begin{proposition}\label{P4}(Energy identity). Let $A$ be a smooth solution to the gradient flow of equation (\ref{E4.6}). For any $t_{1}<t_{2}$, we then have
	\begin{equation}\label{E4}
	CS(A(t_{2}))-CS(A(t_{1}))=-\int_{t_{1}}^{t_{2}}\|F_{A(t)}\wedge\ast_{X} P\|^{2}_{L^{2}(X)}.
	\end{equation}
\end{proposition}
\begin{proof}
Following Bianchi identity $d_{A}F_{A}=0$ and the fact $d\ast_{X}Q=(n-3)\ast_{X}P$, it is easy to check that
	\begin{equation*}
	dtr(F_{A}\wedge F_{A}\wedge\ast_{X}Q)=(n-3)tr(F_{A}\wedge F_{A})\wedge\ast_{X}P)
	\end{equation*}
	We then have
	\begin{equation*}
	\begin{split}
	\frac{\pa }{\pa t}CS(A(t))&=-\frac{1}{n-3}\int_{X}Tr(\frac{\pa F_{A}}{\pa t}\wedge F_{A}\wedge \ast_{X}Q)\\
	&=-\frac{1}{n-3}\int_{X}Trd_{A}(\frac{\pa A}{\pa t})\wedge F_{A}\wedge\ast_{X}Q\\
	&=-\frac{1}{n-3}\int_{X}dTr(\frac{\pa A}{\pa t}\wedge F_{A}\wedge\ast_{X}Q)-\int_{X}Tr(\frac{\pa A}{\pa t}\wedge F_{A}\wedge \ast_{X}P)\\
	&=-\|F_{A}\wedge\ast_{X}P\|^{2}_{L^{2}(X)}.\\
	\end{split}
	\end{equation*}
	Equality (\ref{E4}) follows from integrating the above identity on $[t_{1}, t_{2}]$.
\end{proof}
We define the energy density $\rho(\mathbf{A})$ by
\begin{equation}\label{E3}
\rho(\mathbf{A}):=\lim_{T\rightarrow\infty}\frac{1}{2T}\int_{(-T,T)\times X}|F_{\mathbf{A}}|^{2}dt\wedge dvol_{X}.
\end{equation}
We write $\a\lesssim\b$ to mean that $\a\leq C\b$ for some positive constant $C$ independent of certain parameters on which $\a$ and $\b$ depend. The parameters on which $C$ is independent will be clear or specified at each occurrence. We also use $\b\lesssim\a$ and $\a\approx\b$ analogously.
\begin{lemma}\label{L4.5}
	Let $X$ be a complete manifold of dimension $n$ with a $d$(bounded) $k$-form $\w$, i.e.. there exists a bounded $(k-1)$-form $\theta$ such that $\w=d\theta$, $\a$ be a closed from of degree $n-k$. If $\a$ satisfies
	\begin{equation}\label{E4.7}
	\lim_{r\rightarrow\infty}\frac{1}{r}\int_{B_{r}(x_{0})}|\a|dvol_{X}=0,
	\end{equation} 
	where $x_{0}$ is a point on $X$, $B_{r}(x_{0})$ is a geodesic ball, then there exists a sequence $j_{i}\rightarrow\infty$ as $i\rightarrow\infty$ such that
	\begin{equation*}
	\lim_{i\rightarrow\infty}\int_{B_{j_{i}}(x_{0})}\a\wedge\w=0.
	\end{equation*}
\end{lemma}
\begin{proof}
	Let $\eta:\mathbb{R}\rightarrow\mathbb{R}$ be smooth, $0\leq\eta\leq1$,
	\begin{equation*}
	\eta(t)=\left\{
	\begin{aligned}
	1, &  & t\leq0 \\
	0,  &  & t\geq1
	\end{aligned}
	\right.
	\end{equation*}
	and consider the compactly supported function
	\begin{equation*}
	f_{j}(x)=\eta(\rho(x_{0},x)-j),
	\end{equation*}
	where $j$ is a positive integer and $\rho(x_{0},x)$ stands for the Riemannian distance between $x$ and a base point $x_{0}$.
	
	We consider the form $\b:=\a\wedge\w=d(\a\wedge\theta)$. We have $f_{j}\b=d(f_{j}\a\wedge\theta)-df_{j}\wedge(\a\wedge\theta)$. Following Stokes formula, we obtain that
	\begin{equation*}
	|\int_{X}f_{j}\b|=|\int_{X}df_{j}\wedge(\a\wedge\theta)|\lesssim\int_{B_{j+1}\backslash B_{j}}|\a|.
	\end{equation*}
 Since $\theta$ is bounded, $f_{j}=1$ on $B_{j}$ and $f_{j}=0$ on $X\backslash B_{j+1}$, one obtains that
	\begin{equation*}
	\begin{split}
	|\int_{B_{j}}\b|&=|\int_{B_{j}}f_{j}\b|\leq|\int_{B_{j+1}}f_{j}\b|+|\int_{B_{j+1}\backslash B_{j}}f_{j}\b|\\
	&\leq|\int_{X}f_{j}\b|+\int_{B_{j+1}\backslash B_{j}}|f_{j}\b|\\
	&\lesssim|\int_{X}f_{j}\b|+\int_{B_{j+1}\backslash B_{j}}|\a|.\\
	\end{split}
	\end{equation*}
Thus
	\begin{equation}\label{E4.9}
	|\int_{B_{j}}\b|\lesssim\int_{B_{j+1}\backslash B_{j}}|\a|.
	\end{equation}
By the hypothesis (\ref{E4.7}), there exists a sequence $j_{i}\rightarrow\infty$  as $i\rightarrow\infty$ such that 
	\begin{equation}\label{E4.10}
	\lim_{i\rightarrow\infty}\int_{B_{j_{i}+1}\backslash B_{j_{i}}}|\a|=0.
	\end{equation}
	It now follow (\ref{E4.9})--(\ref{E4.10}) that $\lim_{i\rightarrow\infty}\int_{B_{j_{i}}(x_{0})}\a\wedge\w=0$.
\end{proof}
We denote  by $\ast$ the Hodge star operator on $Cyl(X)$, $D$ by the exterior derivative on $T^{\ast}(Cyl(X))$. We also denote $\tilde{P}=dt\wedge P$, $\tilde{Q}=dt\wedge Q$. Then the forms $\tilde{P}$, $\tilde{Q}$ satisfying 
\begin{equation*}
\ast\tilde{P}=\ast_{X}P,\  \ast\tilde{Q}=\ast_{X}Q,
\end{equation*}
and
\begin{equation*}
D\tilde{P}=4\tilde{Q},\ D\ast\tilde{Q}=(n-3)\ast\tilde{P}.
\end{equation*}
\begin{theorem}\label{T3}
	Let $Cyl(X)$ be the cylinder over a compact real Killing spinor manifold $X$, $\mathbf{A}$ be a smooth solution of instanton equation. If $\rho(\mathbf{A})=0$, then $\mathbf{A}$ is a flat connection.
\end{theorem}
\begin{proof}
	The Yang-Mills energy function is 
	\begin{equation}\nonumber
	\begin{split}
	YM(\mathbf{A}):&=\|F_{\mathbf{A}}\|^{2}_{L^{2}(Z)}=-\int_{\mathbb{R}\times X}Tr(F_{\mathbf{A}}\wedge F_{\mathbf{A}})\wedge\ast\Om\\
	&=-\int_{\mathbb{R}\times X}Tr(F_{\mathbf{A}}^{2})\wedge\ast\tilde{P}-\int_{\mathbb{R}\times X}Tr(F_{\mathbf{A}}^{2})\wedge\ast_{X}Q\wedge dt.\\
	\end{split}
	\end{equation}
	We observe that 
	\begin{equation*}
	-\int_{\mathbb{R}\times X}Tr(F_{\mathbf{A}}^{2})\wedge\ast\tilde{P}=-\frac{1}{n-3}\int_{\mathbb{R}\times X}Tr(F_{\mathbf{A}}^{2})\wedge D\ast\tilde{Q}.
	\end{equation*} 
	Since $\rho(\textbf{A})=0$ and $|Tr(F_{\textbf{A}}^{2})|\lesssim |F_{\textbf{A}}|^{2}$, we observe that
	\begin{equation}\label{E4.1}
	\lim_{T\rightarrow\infty}\frac{1}{T}\int_{(-T,T)\times X}|Tr(F_{\mathbf{A}}^{2})|=0.
	\end{equation} 
Since $Tr(F_{\mathbf{A}}^{2})$ is a closed $4$-form on $Cyl(X)$, it also satisfies Equation (\ref{E4.1}) and $\ast\tilde{P}$ is a $D$(bounded) $(n-4)$-form, following Lemma \ref{L4.5}, there exist a sequence $j_{i}\rightarrow\infty$ as $i\rightarrow\infty$  such that
	\begin{equation}\label{E4.11}
	\lim_{i\rightarrow\infty}\int_{(-j_{i},j_{i})\times X}Tr(F_{\mathbf{A}}^{2})\wedge\ast\tilde{P}=0.
	\end{equation}
	Following equation (\ref{2.1}), we then have
	\begin{equation*}
	-Tr(F_{\mathbf{A}}^{2})\wedge\ast\tilde{P}=-2Tr(\frac{\pa A}{\pa t}\wedge dt\wedge F_{A})\wedge\ast_{X}P=2|\frac{\pa A}{\pa t}|^{2}dt\wedge dvol_{X}
	\end{equation*}
	Thus
	\begin{equation}\label{E4.12}
	-\int_{(-j_{i},j_{i})\times X}Tr(F_{\mathbf{A}}^{2})\wedge\ast\tilde{P}=2\int_{(-j_{i},j_{i})\times X}|\frac{\pa A}{\pa t}|^{2}dt\wedge dvol_{X}.
	\end{equation}
	It now follows (\ref{E4.11}), (\ref{E4.12}) that 
	\begin{equation*}
	\lim_{i\rightarrow\infty}\int_{(-j_{i},j_{i})\times X}|\frac{\pa A}{\pa t}|^{2}dt\wedge dvol_{X}=0,
	\end{equation*}
	i.e., $\frac{\pa A}{\pa t}=0$. The connection $\mathbf{A}$ is not dependent on parameter $t$. Thus
	\begin{equation*}
	\rho(\mathbf{A})=\int_{X}|F_{A}|^{2}dvol_{X}.
	\end{equation*}
	By the hypothesis of energy density $\rho(\textbf{A})$, we obtain that $F_{A}=0$. We complete this proof.
\end{proof}
\begin{corollary}
	Let $Cyl(X)$ be the cylinder over a compact real Killing spinor  manifold $X$, $\mathbf{A}$ be a smooth solution of instanton equation (\ref{E0}). If the curvature $F_{\mathbf{A}}$ is in $L^{p}$, $p\geq2$, then $\mathbf{A}$ is a flat connection.	
\end{corollary}
\begin{proof}
	We denote $B_{T}=(-T,T)\times X$. For $p=2$, it is easy to see $\rho(\mathbf{A})=0$. For $p>2$, we use the H\"{o}lder inequality,
	\begin{equation*}
	\|F_{\mathbf{A}}\|_{L^{2}(B_{T})}\leq\|F_{\mathbf{A}}\|_{L^{p}(B_{
			T})}(2T\rm{Vol}(X))^{\frac{1}{2}-\frac{1}{p}}.
	\end{equation*} 
	Thus $\rho(\mathbf{A})=0$. Following Theorem \ref{T3}, it implies that $\mathbf{A}$ is flat.
\end{proof}
We define $\mathcal{M}_{d}$ as the space of the gauge equivalence classes of instantons $\textbf{A}$ on $P$ satisfying
\begin{equation*}
\|F_{\textbf{A}}\|_{L^{\infty}(X)}\leq d.
\end{equation*}
The space $\mathcal{M}_{d}$ is endowed with the topology of $C^{\infty}$ convergence over compact subsets: the sequence $\{[A_{i}]\}$ in  $\mathcal{M}_{d}$  converges to $[A]$ if and only if there exist gauge transformations $g_{i}$ satisfying $g^{\ast}_{i}(A_{i})\rightarrow A$ in $C^{\infty}$ over every compact subset of $X$. The space $\mathcal{M}_{d}$ is compact by the Uhlenbeck compactness theorem. We denote by $\rho(d)$ the value of $\rho(A)$ over $[A]\in\mathcal{M}_{d}$. At the end of $Z$, there are  connections $\Ga_{\pm}$ over $X$ such that $\mathbf{A}$ converges to $\Ga_{\pm}$, i.e. the restriction $\mathbf{A}|_{X\times\{T\}}$ converges (modulo gauge equivalence) to $\Ga_{\pm}$ in $C^{\infty}$ over $X$ as $T\rightarrow\pm\infty$
\begin{proposition}
	Suppose that $[\textbf{A}]\in\mathcal{M}_{d}$ is an instanton over $Cyl(X)$. Then
	\begin{equation*}
	\rho(d)=\frac{\rho_{-}+\rho_{+}}{2},
	\end{equation*}
	where $\rho_{\pm}=\int_{X}Tr(F_{\Ga_{\pm}}^{2}\wedge\ast_{X}Q)$. 
\end{proposition}
\begin{proof}
	Following equations (\ref{E0}), we have
	\begin{equation*}
	F_{\textbf{A}}^{2}\wedge\ast\Om=F_{A}^{2}\wedge dt\wedge\ast_{X}Q+2\frac{\pa A}{\pa t}\wedge dt\wedge F_{A}\wedge \ast_{X}P.
	\end{equation*}
Following the energy identity in Proposition \ref{P4}, it implies that
	\begin{equation*}
	\int_{(T,T')\times X}|F_{\textbf{A}}|^{2}dt\wedge dvol_{X}
	=\int_{(T,T')}dt\int_{X}Tr(F_{A}^{2}\wedge\ast_{X}Q)+2(CS(A(T))-CS(A(T'))).
	\end{equation*}
	Since $[\textbf{A}]\in\mathcal{M}_{d}$, we have
	$$|CS(A)|=|\frac{1}{2(n-3)}\int_{X}Tr(F_{A}\wedge F_{A})\wedge\ast_{X}Q|\lesssim d^{2}, \forall t\in\mathbb{R}.$$ 
	Thus
	\begin{equation*}
	\rho(d)=\lim_{t\rightarrow\infty}\frac{1}{2T}\int_{(-T,T)\times X}Tr(F_{A}^{2}\wedge\ast_{X}Q)=\frac{\rho_{-}+\rho_{+}}{2}.
	\end{equation*}
	We complete this proof.
\end{proof}
\section*{Acknowledgements}
I would like to thank the anonymous referee for a careful reading of my manuscript and helpful comments. This work is supported by Nature Science Foundation of China No. 11801539 and Postdoctoral Science Foundation of China No. 2017M621998, No. 2018T110616.


\bigskip
\footnotesize
\begin{thebibliography}{SK}
\bibitem{BILL}
Bauer,~I., Ivanova,~T.A., Lechtenfeld,~O., Lubbe,~F.: {\it Yang-Mills instantons and dyons on homogeneous $G_{2}$-manifolds.}
{JHEP.} {\bf2010}(10), 1--27 (2010)
\bibitem{RRC}
Carri{\'o}n,~R.R.: {\it A generalization of the notion of instanton,}
{Diff.Geom.Appl.} {\bf 8}(1), 1--20 (1998)
\bibitem{CDFJ}
Corrigan,~E.~,Devchand,~C.,Fairlie,~D.B., Nuyts.~J.: {\it First order equations for gauge fields in spaces of dimension great than four,}
{Nucl.Phys.B,} {\bf 214}(3), 452--464 (1983)
\bibitem{Donaldson}
Donaldson,~S.K.: {\it Floer homology groups in Yang-Mills theory},
{Cambridge University Press,} (2002)
\bibitem{Donaldson/Kronheimer}
Donaldson S.~K., Kronheimer P.~B.,
The geometry of four-manifolds,
\textit{Oxford University Press}, 1990.
\bibitem{DS}
Donaldson,~S.K., Segal.~E.: {\it Gauge theory in higher dimensions, II.}
{arXiv:0902.3239,} (2009)
\bibitem{DT}
Donaldson,~S.K., Thomas~R.P.: {\it  Gauge theory in higher dimensions,}
{The Geometric Universe, Oxford,} 31--47 (1998)
\bibitem{Gr}
Gra{\~{n}}a,~M.:{\it Flux compactifications in string theory: A comprehensive review.}
{Phys.Rept.} {\bf423}(3), 91--158 (2006)
\bibitem{GSW}
Green,~M.B., Schwarz,~J.H., Witten,~E.: {\it Supperstring theory,}
{Cambridge University Press,} (1987)
\bibitem{HILP}
Harland,~D., Ivanova,~T.A., Lechtenfeld,~O., Popov,~A.D.: {\it Yang-Mills flows on nearly K\"{a}hler manifolds and $G_{2}$-instantons.}
{Comm.Math.Phys.}  {\bf 300}(1), 185--204 (2010)
\bibitem{HN}
Harland,~D., N{\"o}lle~.C.: {\it Instantons and Killing spinors,}
{JHEP.} {\bf3}, 1--38 (2012)
\bibitem{Ha}
Haupt,~A.S.:{\it Yang-Mills solutions and $Spin(7)$-instantons on cylinders over coset spaces with $G_{2}$-structure.}
{JHEP.} {\bf3}, 1--53 (2016)
\bibitem{HT}
Huang,~T.: {\it Instanton on Cylindriacl Manifolds.}
{Ann. Henri Poincar\'{e}} {\bf 18}(2), 623--641 (2017)
\bibitem{HT2}
Huang,~T.: {\it Stable Yang-Mills connections on special holonomy manifolds.}
{J. Geom. Phys.} {\bf116}, 271--280 (2017)
\bibitem{HT3}
Huang,~T: {\it An energy gap for complex Yang-Mills equations. }
{SIGMA} Symmetry Integrability Geom. Methods Appl. {\bf 13} (2017), Paper No. 061, 15 pages. 
\bibitem{ID}
Ivanova,~T.A., Popov,~A.D.: {\it Instantons on special holonomy manifolds.}
{Phys.Rev.D} {\bf 85}(10) (2012)
\bibitem{ILPR}
Ivanova,~T.A., Lechtenfeld,~O., Popov,~A.D., Rahn,~T.: {\it Instantons and Yang-Mills flows on coset spaces.}
{Lett. Math. Phys.} {\bf 89} (3), 231--247 (2009)
\bibitem{Price}
Price,~P.: {\it A monotonicity formula for Yang-Mills fields.}
\textit{Manuscripta Math.} {\bf 43} 131--166 (1983)
\bibitem{Earp2}
S\'{a} Earp,~H.N.: {\it Generalised Chern-Simons Theory and $G_{2}$-Instantons over Associative Fibrations.}
{SIGMA,} {\bf10}:083 (2014)
\bibitem{Earp3}
S\'{a} Earp,~H.N.: {\it $G_{2}$-instantons over asymptotically cylindrical manifolds.}
{Geom. Topol.} {\bf19}, 61--111 (2015)
\bibitem{EW}
S\'{a} Earp,~H.N., Walpuski,~T.:{\it $G_{2}$-instantons over twisted connected sums.}
{Geom. Topol.} {\bf19}, 1263--1285 (2015)
\bibitem{Taubes1990}
Taubes,~C.~H.:{\it Casson's invariant and gauge theory.}
{J. Diff. Geom.}  {\bf31}(2), 547-599 (1990)
\bibitem{Tian}
Tian,~G.:{\it Gauge theory and calibrated geometry, I.}
{Ann. Math.} \textbf{151}(1), 193--268 (2000)
\bibitem{Uhlenbeck1982}
Uhlenbeck,~K.~K.: {\it Connctions with $L^{p}$ bounds on curvature,}
\textit{Comm. Math. Phys.} \textbf{83}, 31--42 (1982)
\bibitem{Uhlenbeck1985}
Uhlenbeck,~K.~K.: {\it The Chern classes of Sobolev connections,}
\textit{Comm. Math. Phys.} \textbf{101}, 445--457 (1985)
\bibitem{Walpuski}
Walpuski,~T.:{\it $G_{2}$-instantons on generalised Kummer constructions.}
{Geom. Topol.} {\bf17}, 2345--2388 (2013)
\bibitem{RSW}
Ward,~R.S.: {\it Completely solvable gauge field equations in dimension great than four,}
{Nucl.Phys.B} {\bf 236}(2), 381--396 (1984)

\baselineskip=17pt
\end{thebibliography}
\end{document}